\documentclass[english,11pt]{amsart}
\usepackage{amsopn,amsthm,amsfonts,amsmath,amssymb,amscd}
\usepackage{babel}
\usepackage{amssymb,tikz,wrapfig,pgfplots}
\usepackage[T1]{fontenc}
\usepackage{bookman} 

\newtheorem{theorem}{theorem}[section]
\newtheorem{lemma}[theorem]{Lemma}

\newtheorem{corollary}[theorem]{Corollary}
\newtheorem{example}[theorem]{Example}

\newtheorem{definition}[theorem]{Definition}

\newenvironment{Proof*}{{\it Proof.}}

\newcommand{\NN}{\mathbb{N}}

\newcommand{\CC}{\mathbb{C}}
\newcommand{\RR}{\mathbb{R}}
\newcommand{\ZZ}{\mathbb{Z}}

\newcommand{\BB}{\mathcal{B}}

\newcommand{\mm}{\mathcal{M}}
\newcommand{\DEF}[1]{\emph{#1}}

    \definecolor{david}{rgb}{.2,.8,.4}
    \definecolor{polona}{rgb}{.5,.5,.2}
   \definecolor{todo}{rgb}{1,.2,.2}

\begin{document}

\title{Cholesky decomposition of matrices over commutative semirings}
\author{David Dol\v zan, Polona Oblak}
\date{\today}

\address{D.~Dol\v zan:~Department of Mathematics, Faculty of Mathematics
and Physics, University of Ljubljana, Jadranska 21, SI-1000 Ljubljana, Slovenia; e-mail: 
david.dolzan@fmf.uni-lj.si}
\address{P.~Oblak: Faculty of Computer and Information Science, University of Ljubljana,
Ve\v cna pot 113, SI-1000 Ljubljana, Slovenia; e-mail: polona.oblak@fri.uni-lj.si}

   \thanks{The authors acknowledge the financial support from the Slovenian Research Agency  (research core funding No. P1-0222)}

\bigskip

\begin{abstract} 
We prove that over a commutative semiring  every symmetric strongly invertible  matrix  with nonnegative numerical range has a Cholesky  decomposition. 
\end{abstract}

\maketitle 

\section{Introduction}

The Cholesky decomposition $A = LL^*$ of a positive semidefinite matrix $A$ over the field of  complex numbers, where $L$ is a lower triangular matrix, is one of the fundamental tools in matrix computations. The standard algorithm for its computation dates from the beginning of the 
previous century and it is one of the most numerically stable matrix factorization algorithms. The
Cholesky decomposition exists for every positive semidefinite matrix. When
assumed $A$ is positive definite and $L$ has all diagonal entries positive, the Cholesky decomposition $A=LL^T$ is unique. 

However, not much is known about the Cholesky decomposition of matrices over semirings. The theory of semirings has many applications in optimization theory, automatic control, models of discrete event networks and graph theory (see e.g. \cite{baccelli1998ergodic,cuninghame2012minimax, li2014heuristic,zhao2010invertible}).

The theory of invertible  matrices over semirings differs from those over complex numbers. The invertible matrices over semirings with 
no non-zero additively invertible elements were characterized in \cite{dolvzan2009invertible}. 
We follow  \cite{MR1473204} and say that a matrix is \DEF{strongly invertible} if all its leading principal submatrices  are invertible. Strongly invertible matrices over semirings were first 
investigated in \cite{tan2017strongly}. The author showed that a matrix over a semiring is strongly invertible 
if and only if it has an LU decomposition, where $L$ and $U$ are both invertible and the diagonal entries of $U$ are invertible elements of the semiring. When applicable over reals, the Cholesky decomposition is roughly twice as efficient as the $LU$ decomposition for solving systems of linear equations (see for example \cite{MR2978290}).

 In this paper, we study the Cholesky decomposition of strongly
invertible matrices over semirings.
In order to construct the Cholesky decomposition, we define the notion of a matrix with a nonnegative numerical range. 
Since positive elements in semirings are not well defined,
we substitute the real case assumption of  $x^TMx$ to be  a positive element with the 
assumption of  $x^TMx$ to be  a square. We show that every strongly invertible 
matrix $M$ with a nonnegative numerical range has a Cholesky
decomposition $M=LL^T$ and $L$ is a symmetric strongly invertible matrix (see Theorem \ref{thm:Ch}). We also show that not every 
symmetric invertible (but not necessarily strongly invertible) matrix with a nonnegative numerical range has a Cholesky decomposition 
(see Example \ref{ex:Z2B}). However, for semirings where sums of squares are squares we prove 
 that a symmetric strongly invertible matrix $M$ has  nonnegative numerical range if and only if it is positive semidefinite, and this holds if and only if it $M$  has a Cholesky decomposition $M=LL^T$, where $L$ is strongly invertible lower triangular matrix (see Corollary \ref{cor:Ch}).

\section{Preliminaries}

A \emph{semiring} is a set $S$ equipped with binary operations $+$ and $\cdot$ such that $(S,+)$ is a commutative monoid with identity element 0, and $(S,\cdot)$ is a monoid with identity element 1. 
In addition, operations $+$ and $\cdot$ are connected by distributivity and 0 annihilates $S$. A semiring is 
\emph{commutative} if $ab=ba$ for all $a,b \in S$.

Throughout the paper we assume that $S$ is a commutative semiring.

The simplest example of a commutative semiring  is the \DEF{binary Boolean semiring} $\BB$, the set $\{0,1\}$
in which $1+1=1\cdot1=1$. Some other examples of  semirings include set of nonnegative integers (or reals) with the usual operations of addition and multiplication, distributive lattices, tropical 
semirings, dio\"{\i}ds, fuzzy algebras, inclines and bottleneck algebras. (See e.g. \cite{MR2389137}.) 

An element $a \in S$ is called \DEF{invertible} if there exists an element $b \in S$ such that $ab = 1$. 
Such an element $b$ is called the \DEF{inverse} of $a$ in $S$ and denoted by $a^{-1}$. It is easily proved that the inverse of $a$ in $S$ is unique. Let $U(S)$ denote the set of all invertible elements in $S$. 

An element $a \in S$ is called \DEF{additively invertible} if $a + b = 0$ for some element
$b$ in $S$. Such an element $b$ is  unique and denoted by $-a$. Let $V(S)$ denote the set of all additively invertible elements in $S$. It is clear that $V(S) = S$ if and only if $S$ is a ring and that $V(S) = \{0\}$ if and only if $0$ is the
only additively invertible element in $S$. 

For a subset $T \subseteq S$,  let $T^n$ denote the set of all vectors of size $n$ over $T$.
We denote by  $\mm_{n}(S)$ the set of all $n \times n$ matrices over $S$. For $M \in \mm_{n}(S)$, we denote by 
$M_{ij}$ the $(i, j)$-entry of $M$, and denote by $M^T$ the transpose of $M$. 
It is easy to see that  $\mm_n(S)$, $n \geq 2$, forms a noncommutative semiring with respect to the matrix addition and the
matrix multiplication. We denote the $m \times n$ zero matrix by $0_{m,n}$  and the $n \times n$ identity matrix by $I_n$.
If the size of a matrix is clear from the context, we omit the subscript denoting it.

\bigskip

Over the complex numbers, a Hermitian matrix $M \in \mm_n(\CC)$ has the Cholesky decomposition
$M=LL^*$ for a lower triangular matrix $L$ if and only if $M$ is positive semidefinite. 
Every positive semidefinite Hermitian matrix 
$M$ has a nonnegative numerical range, i.e.
$ x^* M x \geq 0$ for all $x\in \CC^n$. 
Since the notion of a nonnegative numerical range is defined by utilizing the concept of positive real numbers, we need to introduce
a similar concept to the commutative semiring setting. Therefore, we define the set of squares of a commutative semiring $S$ as
$$Q(S)=\{a\in S; \;a=b^2 \text{ for some } b \in S\}.$$
For $S=\RR$, the set $Q(S)$ coincides with nonnegative real numbers. Now, we can  
define the following.

\begin{definition}
The matrix $M \in \mm_n(S)$ has a \emph{nonnegative numerical range} if 
$x^TMx \in Q(S)$ for all $x \in S^n$.
\end{definition}

We use the following notion of positive semidefiniteness of matrices over semirings, which was defined in \cite{MR3479382,mohindru2014completely}.

\begin{definition}
The symmetric matrix $M \in \mm_n(S)$ is \emph{positive semidefinite} if there exists a matrix $B$ 
such that $M=BB^{T}$.
\end{definition}

Note that the above  two notions do not coincide for symmetric matrices over semirings. For example, it is easy to see that  matrix
$$M=\left[
 \begin{matrix}
  (1,0)&(0,1)\\
  (0,1)&(1,0)
 \end{matrix}
 \right] \in \mm_2(S)$$ over  $S=\ZZ_2 \times \BB$, where $\BB$ is the Boolean semiring, is not
 positive semidefinite, but has a nonnegative numerical range since $Q(S)=S$. (See also Example \ref{ex:Z2B}.)
 On the other hand, the identity matrix over nonnegative integers $\NN$ is positive semidefinite, but does not have a nonnegative numerical range.

 \bigskip

The construction of the Cholesky decomposition will utilize the notion of the Schur complement of a matrix. 
For $M=\left[
 \begin{matrix}
 A &B\\
 C &D
 \end{matrix}
 \right]
  \in \mm_n(\RR)$, the Schur complement of an  invertible submatrix $A$ of $M$ is defined as
  $M/A=D-CA^{-1}B$. 
  For the theory of the Schur complement over real matrices, we refer the reader to  \cite{MR2160825}.  Since  subtraction is generally not possible in commutative semirings, 
we shall only define the Schur complement of  the leading $1\times 1$ submatrix of $M$.

\begin{definition}
Let $$M=\left[
 \begin{matrix}
 a &b^T\\
 b &C
 \end{matrix}
 \right]
  \in \mm_n(S)$$ where $a \in U(S)$, $b \in V(S)^{n-1}$  and $C\in \mm_{n-1}(S)$. We define 
  $$M/a=C+a^{-1}(-b)b^T$$
  to be the \emph{Schur complement} of  $a$ of matrix $M$. 
\end{definition}

The next technical lemma is straightforward.
\begin{lemma}\label{toosimple}
If $a \in U(S)$, $b  \in V(S)^{n-1}$ and $C \in \mm_{n-1}(S)$, then the matrix
$\left[
 \begin{matrix}
 1 &0\\
 a^{-1}b &I
 \end{matrix}
 \right]$ is invertible and
$$
 \left[
 \begin{matrix}
 a &0\\
 0 &M/a
 \end{matrix}
 \right]=\left[
 \begin{matrix}
 1 &0\\
 a^{-1}(-b) &I
 \end{matrix}
 \right]
 \left[
 \begin{matrix}
 a &b^T\\
 b &C
 \end{matrix}
 \right]
 \left[
 \begin{matrix}
 1 &0\\
 a^{-1}(-b) &I
 \end{matrix}
 \right]^T.$$
 \end{lemma}

\begin{proof} 
Observe that
 $\left[
 \begin{matrix}
 1 &0\\
 a^{-1}b &I
 \end{matrix}
 \right]^{-1}=
\left[
 \begin{matrix}
 1 &0\\
 a^{-1}(-b) &I
 \end{matrix}
 \right]$. The second statement of the lemma is a straightforward calculation.
 \end{proof}
 
The next lemma shows that the Schur complement preserves the nonnegative numerical range.

\begin{lemma}\label{pd}
If $a \in U(S)$, $b  \in V(S)^{n-1}$, $C \in \mm_{n-1}(S)$ and $$M=\left[
 \begin{matrix}
 a &b^T\\
 b &C
 \end{matrix}
 \right]
  \in \mm_n(S)$$ has a nonnegative numerical range, then $a$ and $M/a$ both have 
  nonnegative numerical range.
\end{lemma}
\begin{proof}
 Suppose that $M$ has a nonnegative numerical range. Choose $e=[1 \; 0 \ldots 0]^T\in S^n$ and an arbitrary $y \in S^{n-1}$, and observe that 
 $a=e^TMe\in Q(S)$ and
 $$y^T (M/a) y= [0 \; y^T] \left[
 \begin{matrix}
 a &0\\
 0 &M/a
 \end{matrix}
 \right]\left[
 \begin{matrix}
 0\\
 y
 \end{matrix}
 \right].$$
 Define $$x=\left[
 \begin{matrix}
 1 &a^{-1}(-b)^T\\
 0 &I
 \end{matrix}
 \right] \left[
 \begin{matrix}
 0\\
 y
 \end{matrix}
 \right]$$
 and observe that by Lemma \ref{toosimple}, $y^T (M/a) y=x^TMx \in Q(S)$.
%
\end{proof}

\section{Cholesky decomposition}

In this section, we investigate the necessary conditions for a matrix to have a Cholesky decomposition.
Note that not every strongly invertible positive semidefinite matrix over a semiring has a Cholesky decomposition, 
for example
$\left[
 \begin{matrix}
  5 & 2\\
  2 & 1
 \end{matrix}
 \right] \in \mm_2(\ZZ_6)$.
In the next example, we also demonstrate that not every symmetric  invertible matrix with a nonnegative numerical 
range has 
a Cholesky decomposition.

\begin{example}\label{ex:Z2B}
  Let $S=\ZZ_2 \times \BB$, where $\BB$ is the Boolean semiring, and $$M=\left[
 \begin{matrix}
  (1,0)&(0,1)\\
  (0,1)&(1,0)
 \end{matrix}
 \right] \in \mm_2(S).$$  Note that $M^2=I$, so $M$ is invertible. 
 Observe that $Q(S)=S$ and thus $M$ has a nonnegative numerical range, but 
 one can easily see that there does not exist a lower triangular matrix $L$ such that $LL^T=M$. 
 \end{example}

We will therefore use a stronger condition than invertibility, namely the strong invertibility. 
\begin{definition}
The matrix 
$A \in \mm_n(S)$ is \DEF{strongly invertible} if all the leading principal submatrices of $A$ are invertible.
\end{definition}
Note that the matrix $M$ in Example \ref{ex:Z2B} is invertible but not strongly invertible.

The main result of this paper will show that every  symmetric strongly invertible matrix with a nonnegative numerical range  has a Cholesky decomposition $A=LL^T$, which is unique up to a right sided multiplication 
of $L$ by a diagonal matrix  $D$ such that $D^2=I$.

The next lemma proves that every strongly invertible matrix has a Schur complement. It
is a straightforward corollary of \cite[Lemma~2.4]{tan2017strongly}.

\begin{lemma}\label{tan}
 Let $a \in S$, $b  \in S^{n-1}$ and $C \in \mm_{n-1}(S)$. If $M=\left[
 \begin{matrix}
 a &b^T\\
 b &C
 \end{matrix}
 \right]$  is a strongly invertible matrix, then $b$ is additively invertible. 
 Hence, if $M$ is strongly invertible, then  there exists the Schur complement $M/a$.
\end{lemma}

Next, we prove that the set of strongly invertible matrices is invariant under taking the Schur complement.

\begin{lemma}\label{fromd}
  Let $a \in S$, $b  \in S^{n-1}$ and $C \in \mm_{n-1}(S)$. If $M=\left[
 \begin{matrix}
 a &b^T\\
 b &C
 \end{matrix}
 \right]$  is strongly invertible, then $M/a$ is strongly invertible.
\end{lemma}

\begin{proof}
 Choose $k$, $1 \leq k \leq n-1$, and let ${M}_k$ be the leading principal submatrix of $M/a$ of size $k$ and 
 $\hat{M}_k$  the leading principal submatrix of $M$ of size $k+1$. Observe that 
 $\hat{M}_k=\left[
 \begin{matrix}
 a &\hat{b}^T\\
 \hat{b} & \hat{C}
 \end{matrix}
 \right],$ where $\hat{b} \in S^{k-1}$ consists of the first $k-1$ components of the vector $b$ and 
 $\hat{C} \in \mm_{k-1}(S)$ is the leading principal submatrix of $C$.
Furthermore, $ {M}_k$ is the Schur complement of $\hat{M}_k$.
Since $M$ is strongly invertible, we conclude that the matrix $\hat{M}_k$ is invertible, and thus by Lemma \ref{toosimple} it follows that
$\left[
 \begin{matrix}
 a &0\\
 0 &\hat{M}_k/a
 \end{matrix}
 \right]$
 is invertible. This implies that $\hat{M}_k \in U(\mm_{k}(S))$, hence $M/a$ is strongly invertible.
 \end{proof}

\begin{theorem}[Cholesky decomposition]\label{thm:Ch}
 Let $M\in \mm_n(S)$ be a symmetric  strongly invertible matrix with a nonnegative numerical range. 
 Then $M=LL^T$ for a strongly invertible lower triangular
 matrix $L \in \mm_n(S)$.
 
 Furthermore, the above decomposition $M=LL^T$ is unique up to right sided multiplication of $L$ by a diagonal matrix  $D \in \mm_n(S)$ such that $D^2=I$.
\end{theorem}

\begin{proof}
  Suppose $M=\left[
 \begin{matrix}
 a &b^T\\
 b &C
 \end{matrix}
 \right]\in \mm_n(S)$, where $a \in S$, $b  \in S^{n-1}$ and $C \in \mm_{n-1}(S)$, is  strongly invertible 
 and has a nonnegative numerical range. We will prove the existence of $L$ by the induction on $n$. For $n=1$, the statement is clear. Suppose $n >1$.  Since $M$
 is strongly invertible, we have $a \in U(S)$. By Lemmas \ref{pd}, \ref{tan} and \ref{fromd}, the
 Schur complement $M/a \in \mm_{n-1}(S)$ exists, it is strongly invertible and has a nonnegative numerical range. 
 Thus by the induction hypothesis, there exists a strongly invertible  lower triangular matrix $K \in \mm_{n-1}(S)$ 
 such that $M/a=K K^T$. 
Moreover, Lemma \ref{pd} implies that $a$  has a nonnegative numerical range and thus by definition, there
exists $k \in S$ such that $a=k^2$.  Note that $a\in U(S)$ also implies $k\in U(S)$. If
 $L=\left[
 \begin{matrix}
 k &0\\
 k^{-1} b &K
 \end{matrix}
 \right]$, then $L L^T=M$. Since $K$ is a lower triangular strongly invertible  matrix, 
 it follows directly that $L$ is strongly invertible .
 
 Moreover, if $M=LL^T=\tilde{L}\tilde{L}^T$, where $L$ and $\tilde{L}^T$ are invertible  lower triangular matrices, then 
 $$L^{-1}\tilde{L}=L^{T}(\tilde{L}^{T})^{-1}=(\tilde{L}^{-1}L)^T.$$
 This implies that $L^{-1}\tilde{L}=(\tilde{L}^{-1}L)^T$ is a diagonal matrix. We denote  
 $D=L^{-1}\tilde{L}$ and so
 $L=\tilde{L}D$ and $\tilde{L}=LD$, where $D^2=I$.
 Obviously, for any diagonal matrix $D$ such that $D^2=I$, we have $(LD)(LD)^T=LL^T$.
\end{proof}

The next example shows that in general, the lower triangular matrix $L$ from Theorem \ref{thm:Ch} cannot be chosen to have a nonnegative numerical range.

\begin{example}
  Let $S=\ZZ_2[x]/(x^3)$ be a (semi)ring. 
	Observe that $Q(S)=\{0,1,x^2,1+x^2\}$.
	Choose $$M=\left[
 \begin{matrix}
  1&0\\
  0&1+x^2
 \end{matrix}
 \right] \in \mm_2(S).$$  Note that $M$ is strongly invertible.
 Since $1+x^2 \in Q(S)$ and $(a+b)^2=a^2+b^2$ for all $a, b \in S$, $M$ has a nonnegative numerical range.  
 One can check that for $$L=\left[
 \begin{matrix}
  1&0\\
  0&1+x
 \end{matrix}
 \right] \in \mm_2(S)$$
 we have $LL^T=M$, but $L$ does not have a nonnegative numerical range, since $1+x \notin Q(S)$. 
 Now, if $D$ is any diagonal matrix with $D^2=I$, then the diagonal entries of $D$ are either equal to $1$ or $1+x^2$.
 Since both of these two elements are from $Q(S)$ and $1+x \notin Q(S)$, the matrix $LD$ cannot have a nonnegative numerical range.
 \end{example}

The next corollary shows that in some  commutative semirings, we can obtain a characterization of positive semidefinite strongly invertible matrices by their Cholesky decomposition. 

\begin{corollary}\label{cor:Ch}
 Let $S$ be a commutative semiring such that $Q(S)+Q(S)\subseteq Q(S)$ and let $M \in \mm_n(S)$.  The following statements are equivalent.
 \begin{enumerate}
  \item\label{1} $M$ is a symmetric strongly invertible matrix with a  nonnegative numerical range.
  \item\label{2} $M=LL^T$ for a strongly invertible lower triangular matrix $L \in \mm_n(S)$.
  \item\label{3} $M$ is a symmetric strongly invertible positive semidefinite matrix.
 \end{enumerate}
 \end{corollary}

\begin{proof}
By Theorem \ref{thm:Ch}, we have that (\ref{1}) implies (\ref{2}). If $M=LL^T$, then every leading 
principal submatrix of $M$ is a product of the corresponding leading principal submatrix $K$ of $L$ 
with its transpose $K^T$. Moreover, if $L$ is strongly invertible, it follows that $M$ is strongly invertible and thus  (\ref{2}) implies (\ref{3}).
Assume now that $M=BB^T \in \mm_n(S)$ is strongly invertible. Then 
$$x^TMx=(B^Tx)^T(B^Tx) \in \sum Q(S) \subseteq Q(S)$$
and thus  (\ref{3}) implies (\ref{1}).
\end{proof}

%
%

  \section*{Acknowledgements}
	The authors acknowledge the financial support from the Slovenian Research Agency  (research core funding No. P1-0222).


\bibliographystyle{plain}
\bibliography{semiring}

\end{document}